\documentclass[oneside,english]{amsart}
\usepackage[T1]{fontenc}
\usepackage[latin9]{inputenc}
\usepackage{amsthm}
\usepackage{amstext}

\makeatletter

\newcommand{\lyxmathsym}[1]{\ifmmode\begingroup\def\b@ld{bold}
  \text{\ifx\math@version\b@ld\bfseries\fi#1}\endgroup\else#1\fi}

\numberwithin{equation}{section}
\numberwithin{figure}{section}
\theoremstyle{plain}
\newtheorem{thm}{\protect\theoremname}
\ifx\proof\undefined
\newenvironment{proof}[1][\protect\proofname]{\par
\normalfont\topsep6\p@\@plus6\p@\relax
\trivlist
\itemindent\parindent
\item[\hskip\labelsep\scshape #1]\ignorespaces
}{%
\endtrivlist\@endpefalse
}
\providecommand{\proofname}{Proof}
\fi
\theoremstyle{plain}
\newtheorem{cor}[thm]{\protect\corollaryname}

\makeatother

\usepackage{babel}
\providecommand{\corollaryname}{Corollary}
\providecommand{\theoremname}{Theorem}

\begin{document}

\title{On some general solutions of the simple Pell equation}

\author{Vladimir Pletser}

\address{European Space Research and Technology Centre, ESA-ESTEC P.O. Box
299, NL-2200 AG Noordwijk, The Netherlands E-mail: Vladimir.Pletser@esa.int }
\begin{abstract}
Two theorems are demonstrated giving analytical expressions of the
fundamental solutions of the Pell equation $X^{2}-DY^{2}=1$ found
by the method of continued fractions for two squarefree polynomial
expressions of radicands of Richaud-Degert type $D$ of the form $D=\left(f\left(u\right)\right)^{2}\pm2^{\alpha}n$,
where $D$, $n>0$, $\alpha\geq0,\in\mathbb{Z}$, and $f\left(u\right)>0,\in\mathbb{Z}$,
any polynomial function of $u\in\mathbb{Z}$ such that $f\left(u\right)\equiv0\left(mod\,\left(2^{\alpha-1}n\right)\right)$
or {\normalsize{}$f\left(u\right)\equiv\left(2^{\alpha-2}n\right)\left(mod\,\left(2^{\alpha-1}n\right)\right)$.}{\normalsize \par}

\textbf{Keywords}: Continued fraction development (11J70; 11Y65);
Solutions of Pell equation (11D09)
\end{abstract}

\maketitle

\section{\noindent Introduction}

\noindent Pell equations of the general form

\noindent 
\begin{equation}
X^{2}-DY^{2}=N\label{eq:3-1-1}
\end{equation}
with $X,Y,N\in\mathbb{Z}$ and squarefree $D>0,\in\mathbb{Z}$, have
been investigated in various forms since long, already by Indian and
Greek mathematicians and further in the 17th and 18th centuries by
Pell, Brouckner, Lagrange, Euler and others (see historical accounts
in \cite{key-6}, \cite{key-8}%
\footnote{\cite{key-8} lists 138 references of articles on Pell equations from
1658 to 1943.%
}, \cite{key-2}, \cite{key-5}) and are treated in several classical
text books (see e.g. \cite{key-1a}, \cite{key-7}, \cite{key-3},
\cite{key-1} and references therein). 

\noindent For $N=1$, the simple Pell equation reads classically 

\noindent 
\begin{equation}
X^{2}-DY^{2}=1\label{eq:3-1}
\end{equation}
Beside the obvious trivial solution $(X_{t},Y_{t})=(1,0)$, a whole
infinite branch of solutions exists for $n>0,\in\mathbb{\mathbb{Z}}$
given by
\begin{eqnarray}
X_{n} & = & \frac{\left(X_{1}+\sqrt{D}Y_{1}\right)^{n}+\left(X_{1}-\sqrt{D}Y_{1}\right)^{n}}{2}\nonumber \\
Y_{n} & = & \frac{\left(X_{1}+\sqrt{D}Y_{1}\right)^{n}-\left(X_{1}-\sqrt{D}Y_{1}\right)^{n}}{2\sqrt{D}}\label{eq:5-1}
\end{eqnarray}
where $(X_{1},Y_{1})$ is the fundamental solution to (\ref{eq:3-1}),
i.e. the smallest integer solution ($X_{1}>1,Y_{1}>0,\in\mathbb{Z}^{+}$)
different from the trivial solution. Robertson \cite{key-9} lists
five methods to find the fundamental solution $(X_{1},Y_{1})$. Among
these, the classical method introduced by Lagrange \cite{key-10},
based on the continued fraction expansion of the quadratic irrational
$\sqrt{D}$, is at the heart of several other methods (see also \cite{key-11})
.

\noindent In this paper, we demonstrate two theorems yielding analytical
expressions of the fundamental solutions $(X_{1},Y_{1})$ of the simple
Pell equation for squarefree polynomial expressions of radicands of
Richaud-Degert type $D$ of the form $D=\left(f\left(u\right)\right)^{2}\pm2^{\alpha}n$
where $D,n\in\mathbb{Z}^{+}$, $\alpha\geq0\in\mathbb{Z}^{*}$ and
$f\left(u\right)$ any polynomial function of $u\in\mathbb{Z}$ with
$f\left(u\right)>0,\in\mathbb{Z}^{+}$ such that first, $f\left(u\right)\equiv0\left(mod\,\left(2^{\alpha-1}n\right)\right)$
and second, $f\left(u\right)\equiv\left(2^{\alpha-2}n\right)\left(mod\,\left(2^{\alpha-1}n\right)\right)$.

\section{\noindent Solutions of the Pell equation by the continued fraction
method}

\noindent Quadratic irrationals $x$ of the form 
\begin{equation}
x=\frac{\sqrt{D}+b}{c}\label{eq:1-1}
\end{equation}
where $b$, $c\neq0$ and $D>0,\in\mathbb{Z}$, $x\in\mathbb{R}$,
and $D$ is not a perfect square ($\sqrt{D}\notin\mathbb{\mathbb{Z}}$),
have been studied extensively (see e.g. \cite{key-12}, \cite{key-13},
\cite{key-14}, \cite{key-4}). It is known also that a real number
$x$ has an eventually periodic regular continued fraction expansion
if and only if $x$ is a quadratic irrational. It yields that for
every squarefree $D$, there exists $r>0,\in\mathbb{Z}^{+}$ such
that the regular continued fraction expansion of $\sqrt{D}$ is given
by 
\begin{equation}
\sqrt{D}=[a_{0};\overline{a_{1},\text{\dots},a_{r},2a_{0}}]\label{eq:2-1}
\end{equation}
where $a_{0}=\left\lfloor \sqrt{D}\right\rfloor $, the greatest integer
$\leq\sqrt{D}$; the $a_{i}$ ($1\leq i\leq r$, where $i$ and $a_{i}>0,\in\mathbb{Z}^{+}$)
are the partial denominators of the continued fraction expansion;
the finite sequence $a_{1},\text{\dots},a_{r}$ is symmetric (so $a_{j}=a_{r-j+1}$
for $1\leq j\leq\left\lfloor r/2\right\rfloor ,$ with $j>0,\in\mathbb{Z}^{+}$);
and the bar indicates the period of length $r+1$. We use then the
algorithm given by Sierpinski \cite{key-4} to calculate the continued
fraction development of $\sqrt{D}$, that can be summarized as follows.
Let $b_{i}$ and $c_{i}>0,\in\mathbb{Z}$, $x_{i}\in\mathbb{R}$ irrational
numbers, $a_{i}=\left\lfloor x_{i}\right\rfloor $ for $i\geq1$,
the partial denominators $a_{i}$ of the continued fraction are then
successively
\begin{equation}
a_{0}=\left\lfloor \sqrt{D}\right\rfloor \Longrightarrow\sqrt{D}=a_{0}+\frac{1}{x_{1}}=a_{0}+\frac{c_{1}}{\sqrt{D}+b_{1}},\,\,\, b_{1}=a_{0},\,\,\, c_{1}=D-b_{1}^{2}\label{eq:3-1-2}
\end{equation}
\[
a_{i}=\left\lfloor x_{i}\right\rfloor \Longrightarrow x_{i}=a_{i}+\frac{1}{x_{i+1}}=a_{i}+\left(\frac{\sqrt{D}+b_{i}}{c_{i}}-a_{i}\right)=a_{i}+\frac{c_{i+1}}{\sqrt{D}+b_{i+1}},
\]
\begin{equation}
b_{i+1}=a_{i}c_{i}-b_{i},\,\,\, c_{i+1}=\frac{D-b_{i+1}^{2}}{c_{i}}\label{eq:4-2}
\end{equation}
(see \cite{key-4}, p.313, for more details). Therefore $a_{i-1}$,
$b_{i}$ and $c_{i}$ can be calculated successively by 
\begin{equation}
a_{i-1}=\left\lfloor \frac{a_{0}+b_{i-1}}{c_{i-1}}\right\rfloor ,\,\,\, b_{i}=a_{i-1}c_{i-1}-b_{i-1},\,\,\, c_{i}=\frac{D-b_{i}^{2}}{c_{i-1}}\label{eq:5-1-1}
\end{equation}
with $b_{0}=0$ and $c_{0}=1$, until $a_{i}$ is found such as $a_{i}=a_{r+1}=2a_{0}$. 

\noindent The fundamental solution $(X_{1},Y_{1})$ of the simple
Pell equation (\ref{eq:3-1}) is found for a particular value of $D$
by computing the $r$\textsuperscript{th}convergent $\left(p_{r}/q_{r}\right)$
of the continued fraction $[a_{0};\overline{a_{1},\text{\dots},a_{r},2a_{0}}]$
of $\sqrt{D}$. The convergents $p_{i}$ and $q_{i}$ can be found
by the recurrence relations
\begin{equation}
p_{i}=a_{i}p_{i-1}+p_{i-2}\,\,;\,\,\,\,\, q_{i}=a_{i}q_{i-1}+q_{i-2}\label{eq:4}
\end{equation}
with $p_{-1}=1$, $p_{0}=a_{0}$, $q_{-1}=0$, $q_{0}=1$. The fundamental
solution $(X_{1},Y_{1})$ is then $(p_{r},q_{r})$ if $r\equiv1(mod\,2)$,
or $(p_{2r+1},q_{2r+1})$ if $r\equiv0(mod\,2)$, with $X_{1}$ and
$Y_{1}$ forming an irreducible fraction $\left(X_{1}/Y_{1}\right)$,
where, to recall, $r+1$ is the period length of the continued fraction
expansion of $\sqrt{D}$.

\section{\noindent Theorems on general solutions of the simple Pell equation}

\noindent A first theorem is demonstrated allowing to find the convergents
$\left(p_{r}/q_{r}\right)$ of $\sqrt{D}$ for $D=\left(f\left(u\right)\right)^{2}\pm2^{\alpha}n$
where $n>0$ and $\alpha\geq0,\in\mathbb{Z}$, and $f\left(u\right)$
is any polynomial function of $u\in\mathbb{\mathbb{Z}}$, taking only
positive integer values and such that $f\left(u\right)\equiv0\left(mod\,\left(2^{\alpha-1}n\right)\right)$
for $\forall u\in\mathbb{\mathbb{Z}}$.
\begin{thm}
\noindent For squarefree $D$, $n$ and $r>0,\in\mathbb{Z}$ and $\alpha\geq0,\in\mathbb{Z}$,
let $f\left(u\right)$ be any polynomial function of $u\in\mathbb{Z}$
such that $f\left(u\right)>0,\in\mathbb{Z}$ for $\forall u\in\mathbb{\mathbb{Z}}$
and $f$ is written instead of $f\left(u\right)$ for convenience;
if 
\begin{equation}
D=f^{2}\pm2^{\alpha}n\label{eq:6}
\end{equation}
and 
\begin{equation}
f\equiv0\left(mod\,\left(2^{\alpha-1}n\right)\right)\label{eq:6-1}
\end{equation}
then the convergents $\left(p_{r}/q_{r}\right)$ of $\sqrt{D}$ are
\begin{equation}
\frac{p_{r}}{q_{r}}=\left(\frac{\frac{f^{2}}{2^{\alpha-1}n}\pm1}{\frac{f}{2^{\alpha-1}n}}\right)\label{eq:7}
\end{equation}
and for the case of the minus sign in front of $2^{\alpha}n$ in (\ref{eq:6}),
$f^{2}>2^{\alpha}n$ and $\alpha$ and $n$ not taking simultaneously
the values $\alpha=0$ and $n=1$.
\end{thm}
\noindent This theorem is demonstrated by calculating the continued
fraction development of $\sqrt{D}$ following Sierpinski's algorithm
(\ref{eq:3-1-2}) to (\ref{eq:5-1-1}) and the convergents $\left(p_{r}/q_{r}\right)$
of $\sqrt{D}$ by (\ref{eq:4}).
\begin{proof}
\noindent For squarefree $D$, $n$ and $r>0,\in\mathbb{Z}$; $\alpha\geq0,\in\mathbb{Z}$;
$f\left(u\right)$ any polynomial of $u\in\mathbb{Z}$ taking values
$f=f\left(u\right)>0,\in\mathbb{Z}$.

\noindent (i) For the case of the plus sign in front of $2^{\alpha}n$
in (\ref{eq:6}), i.e.
\begin{equation}
D=f^{2}+2^{\alpha}n\label{eq:8}
\end{equation}
one has
\begin{equation}
f^{2}<f^{2}+2^{\alpha}n<\left(f+1\right)^{2}\label{eq:9}
\end{equation}
under the condition that
\begin{equation}
f\ge2^{\alpha-1}n\label{eq:10}
\end{equation}
It yields successively $a_{0}=\left\lfloor \sqrt{f^{2}+2^{\alpha}n}\right\rfloor =f$,
$b_{1}=f$, $c_{1}=2^{\alpha}n$; $a_{1}=\frac{f}{2^{\alpha-1}n}$
with the condition, if $\alpha>0$,
\begin{equation}
f\equiv0\left(mod\,\left(2^{\alpha-1}n\right)\right)\label{eq:11}
\end{equation}
or if $\alpha=0$,
\begin{equation}
2f\equiv0\left(mod\, n\right)\label{eq:12}
\end{equation}
 It yields then $b_{2}=f$, $c_{2}=1$; $a_{2}=2f=2a_{0}$, giving
the continued fraction
\begin{equation}
\sqrt{f^{2}+2^{\alpha}n}=\left[f;\left(\frac{f}{2^{\alpha-1}n}\right),2f,\text{\dots}\right]\label{eq:13}
\end{equation}
with $r=1$, which, with (\ref{eq:4}), yields immediately 
\begin{equation}
\frac{p_{r}}{q_{r}}=\left(\frac{\frac{f^{2}}{2^{\alpha-1}n}+1}{\frac{f}{2^{\alpha-1}n}}\right)\label{eq:7-1}
\end{equation}

\noindent (ii) For the case of the minus sign in front of $2^{\alpha}n$
in (\ref{eq:6}), i.e.
\begin{equation}
D=f^{2}-2^{\alpha}n\label{eq:15}
\end{equation}
and $\alpha$ and $n$ not simultaneously $\alpha=0$ and $n=1$,
one has
\begin{equation}
\left(f-1\right)^{2}<f^{2}-2^{\alpha}n<f^{2}\label{eq:16}
\end{equation}
under the condition that
\begin{equation}
f>2^{\alpha-1}n\label{eq:17}
\end{equation}
It yields successively $a_{0}=\left\lfloor \sqrt{f^{2}-2^{\alpha}n}\right\rfloor =\left(f-1\right)$,
$b_{1}=\left(f-1\right)$, 

\noindent $c_{1}=\left(2f-2^{\alpha}n-1\right)$; $a_{1}=1$ with
the condition
\begin{equation}
f>2^{\alpha}n\label{eq:18}
\end{equation}
One has then $b_{2}=\left(f-2^{\alpha}n\right)$, $c_{2}=2^{\alpha}n$;
$a_{2}=\left(\frac{f}{2^{\alpha-1}n}-2\right)$ with the condition
(\ref{eq:11}) if $\alpha>0$ or (\ref{eq:12}) if $\alpha=0$. It
yields then $b_{3}=\left(f-2^{\alpha}n\right)$, $c_{3}=\left(2f-2^{\alpha}n-1\right)$;
$a_{3}=1$, $b_{4}=\left(f-1\right)$, $c_{4}=1$; $a_{4}=2\left(f-1\right)=2a_{0}$,
giving the continued fraction
\begin{equation}
\sqrt{f^{2}-2^{\alpha}n}=\left[\left(f-1\right);1,\left(\frac{f}{2^{\alpha-1}n}-2\right),1,2\left(f-1\right),\text{\dots}\right]\label{eq:19}
\end{equation}
with $r=3$, which, with (\ref{eq:4}), yields immediately
\begin{equation}
\frac{p_{r}}{q_{r}}=\left(\frac{\frac{f^{2}}{2^{\alpha-1}n}-1}{\frac{f}{2^{\alpha-1}n}}\right)\label{eq:7-1-1}
\end{equation}

\end{proof}
\noindent Note that the conditions (\ref{eq:6-1}) and $f>0$ means
that $\exists\, k>0,\in\mathbb{Z}$ such that $f=2^{\alpha-1}kn$,
yielding simpler expressions of $D$ and of the convergents $\left(p_{r}/q_{r}\right)$
of $\sqrt{D}$ as
\begin{equation}
D=2^{\alpha}n\left(2^{\alpha-2}k^{2}n\pm1\right)\label{eq:7-2}
\end{equation}

\noindent 
\begin{equation}
\frac{p_{r}}{q_{r}}=\left(\frac{2^{\alpha-1}k^{2}n\pm1}{k}\right)\label{eq:7-3}
\end{equation}

\noindent The following corollary give particular relations deduced
from this Theorem, some of them being already given by other authors
(see e.g. \cite{key-4}).
\begin{cor}
\noindent For squarefree $D$, $d$, $m$, $n$ and $r>0,\in\mathbb{Z}$;
$\alpha\geq0,\in\mathbb{Z}$; and for appropriate conditions as in
Theorem 1, the following relations hold for the convergents $\left(p_{r}/q_{r}\right)$
of $\sqrt{D}$ : 

\noindent (i) for $D=\left(d^{2}-1\right)$,
\begin{equation}
\frac{p_{r}}{q_{r}}=\left(\frac{d}{1}\right)\label{eq:21}
\end{equation}
(ii) for $D=\left(d^{2}+1\right)$, 
\begin{equation}
\frac{p_{r}}{q_{r}}=\left(\frac{2d^{2}+1}{2d}\right)\label{eq:22}
\end{equation}
 (iii) for $D=\left(d^{2}\pm2\right)$,
\begin{equation}
\frac{p_{r}}{q_{r}}=\left(\frac{d^{2}\pm1}{d}\right)\label{eq:23}
\end{equation}
(iv) for$D=m^{2}\left(d^{2}\pm2^{\alpha}n\right)$, 
\begin{equation}
\frac{p_{r}}{q_{r}}=\left(\frac{\frac{d^{2}}{2^{\alpha-1}n}\pm1}{\frac{d}{2^{\alpha-1}mn}}\right)\label{eq:24}
\end{equation}
(v) for $D=\left(\left(md^{\beta}\right)^{2}\pm2^{\alpha}n\right)$,
\begin{equation}
\frac{p_{r}}{q_{r}}=\left(\frac{\frac{\left(md^{\beta}\right)^{2}}{2^{\alpha-1}n}\pm1}{\frac{md^{\beta}}{2^{\alpha-1}n}}\right)\label{eq:25}
\end{equation}
(vi) for$D=n\left(nd^{2}\pm1\right)$, 
\begin{equation}
\frac{p_{r}}{q_{r}}=\left(\frac{2nd^{2}\pm1}{2d}\right)\label{eq:26}
\end{equation}
(vii) for$D=d\left(m^{2}d\pm2\right)$,
\begin{equation}
\frac{p_{r}}{q_{r}}=\left(\frac{m^{2}d\pm1}{m}\right)\label{eq:28}
\end{equation}
(viii) for $D=d\left(d+4\right)$, $d\equiv0\left(mod\,2\right)$
and $d\geq2$,
\begin{equation}
\frac{p_{r}}{q_{r}}=\left(\frac{\frac{\left(d+2\right)^{2}}{2}-1}{\frac{d+2}{2}}\right)\label{eq:29}
\end{equation}
(ix) for $D=\left(\left(d\pm2\right)^{2}\pm4\right)$, $d\equiv0\left(mod\,2\right)$
and $d\geq2$,
\begin{equation}
\frac{p_{r}}{q_{r}}=\left(\frac{\frac{\left(d\pm2\right)^{2}}{2}\pm1}{\frac{d\pm2}{2}}\right)\label{eq:30}
\end{equation}
\end{cor}
\begin{proof}
\noindent For squarefree $D$, $d$, $d'$, $m$, $n$, $n'$ and
$r>0,\in\mathbb{Z}$; $t$, $\alpha$ and $\beta\geq0,\in\mathbb{Z}$;
let 
\begin{equation}
f\left(d\right)=\left(md^{\beta}\pm t\right)\label{eq:30-1}
\end{equation}
 be a polynomial of $d$ such that $f\left(d\right)>0,\in\mathbb{Z}$.
One has then the following from Theorem 1, yielding the respective
convergents by (\ref{eq:4}): 

\noindent (i) let $t=\alpha=0$, and $m=n=\beta=1$ in (\ref{eq:30-1})
and (\ref{eq:6}), yielding $f\left(d\right)=d$ and $D=d^{2}-1$.
Then, in the proof (ii) of Theorem 1, $a_{2}=2\left(d-1\right))=2a_{0}$
for $d\geq2$, and the continued fraction reduces to $\sqrt{d{}^{2}-1}=\left[\left(d-1\right);1,2\left(d-1\right),\ldots\right]$,
with $r=1$, yielding directly (\ref{eq:21}) by (\ref{eq:4})%
\footnote{In \cite{key-4} p.321, it is erroneously stated that \textquotedblleft $\sqrt{a{}^{2}+1}=\left[a;a,2a,\ldots\right]$
(instead of $\left[a;2a,\text{\dots}\right]$) for any natural number
a\textquotedblright{} and \textquotedblleft $\sqrt{\left(na{}^{2}\right)^{2}+1}$
(instead of $\sqrt{\left(na\right)^{2}+1}$ ) $=\left[\left(na-1\right);1,2n-2,1,2\left(na-1\right),\text{\dots}\right]$
for natural numbers $a$ and $n>1$\textquotedblright{}%
}. Note that (\ref{eq:7}) yields a following convergent for $r\lyxmathsym{\textquoteright}=2r+1=3$,
which provides a next solution to the Pell equation. 

\noindent (ii) let $t=\alpha=0$, and $m=n=\beta=1$ in (\ref{eq:30-1})
and (\ref{eq:6}), yielding $f\left(d\right)=d$ and $D=d^{2}+1$;
then (\ref{eq:7}) yields immediately (\ref{eq:22}); 

\noindent (iii) let $t=0$, and $m=n=\alpha=\beta=1$ in (\ref{eq:30-1})
and (\ref{eq:6}) yielding $f\left(d\right)=d$ and $D=d^{2}\pm2$,
giving immediately (\ref{eq:23}) from (\ref{eq:7}), with, for the
minus sign, the condition that $d\geq3$; if $d=2$, the continued
fraction (\ref{eq:19}) reduces to $\sqrt{2}=\left[1;2,\ldots\right]$
with $r=1$, yielding $\left(p_{1}/q_{1}\right)=\left(3/2\right)$
by (\ref{eq:4}). 

\noindent (iv) let $t=0$, $\beta=1$ and $n=n'm^{2}$ in (\ref{eq:30-1})
and (\ref{eq:6}), yielding $f\left(d\right)=md$ and $D=m^{2}\left(d^{2}\pm2^{\alpha}n'\right)$
and dropping the $'$; then (\ref{eq:24}) is immediate from (\ref{eq:7})
with the condition that $d\equiv0\left(mod\,\left(2^{\alpha-1}nm\right)\right)$;

\noindent (v) let $t=0$ in (\ref{eq:30-1}), yielding $f\left(d\right)=md^{\beta}$;
then (\ref{eq:25}) is immediate from (\ref{eq:7}), with the condition
that $md^{\beta}\equiv0\left(mod\,\left(2^{\alpha-1}n\right)\right)$;

\noindent (vi) setting $t=\alpha=0$, $\beta=1$ and $m=n$ in (\ref{eq:25})
yields immediately (\ref{eq:26});

\noindent (vii) let $t=0$ and $\beta=1$ and either $\alpha=1$ and
$n=d$, or $\alpha=0$ and $n=2d$, then (\ref{eq:28}) is immediate
from (\ref{eq:7}); 

\noindent (viii) and (ix) let $m$, $n$ and $\beta=1$, $t$ and
$\alpha=2$, and $d\equiv0\left(mod\,2\right)$ with $d\geq2$; then
(\ref{eq:29}) and (\ref{eq:30}) are immediate from (\ref{eq:7}). 
\end{proof}
\noindent Another theorem is demonstrated allowing to calculate the
convergents $\left(p_{r}/q_{r}\right)$ of $\sqrt{D}$ for $D=\left(f\left(u\right)\right)^{2}\pm2^{\alpha}n$
where $f\left(u\right)$ is any polynomial function of $u$, taking
only positive integer values and such that $f\left(u\right)\equiv\left(2^{\alpha-2}n\right)\left(mod\,\left(2^{\alpha-1}n\right)\right)$
for $\forall u\in\mathbb{\mathbb{Z}}$, with $n>0$ and $\alpha\geq0,\in\mathbb{Z}$. 
\begin{thm}
\noindent For squarefree $D$, $n$ and $r>0,\in\mathbb{Z}$ and $\alpha\geq0,\in\mathbb{Z}$,
let $f\left(u\right)$ be any polynomial function of $u\in\mathbb{Z}$
such that $f\left(u\right)>0,\in\mathbb{Z}$ for $\forall u\in\mathbb{\mathbb{Z}}$
and $f$ is written instead of $f\left(u\right)$ for convenience;
if
\begin{equation}
D=f^{2}\pm2^{\alpha}n\label{eq:31}
\end{equation}
and
\begin{equation}
f\equiv\left(2^{\alpha-2}n\right)\left(mod\,\left(2^{\alpha-1}n\right)\right)\label{eq:6-1-1}
\end{equation}
then the convergents $\left(p_{r}/q_{r}\right)$ of $\sqrt{D}$ are,
if $f\equiv0\left(mod\,2\right)$,
\begin{equation}
\frac{p_{r}}{q_{r}}=\left(\frac{\frac{f^{2}\left(f^{2}\text{\ensuremath{\pm}}2^{\alpha}n\right)}{2^{2\alpha-3}n^{2}}+1}{\frac{f\left(f^{2}\text{\ensuremath{\pm}}2^{\alpha-1}n\right)}{2^{2\alpha-3}n^{2}}}\right)\label{eq:32}
\end{equation}
and, if $f\equiv1\left(mod\,2\right)$, 
\begin{equation}
\frac{p_{r}}{q_{r}}=\left(\frac{\frac{f^{2}\left(f^{2}\pm3\times2^{\alpha-2}n\right)^{2}}{2^{3\alpha-5}n^{3}}\pm1}{\frac{f\left(f^{2}\pm2^{\alpha-2}n\right)\left(f^{2}\pm3\times2^{\alpha-2}n\right)}{2^{3\alpha-5}n^{3}}}\right)\label{eq:33}
\end{equation}
under the conditions for $D=\left(f{}^{2}+2^{\alpha}n\right)$ and
$f\equiv0$ or $1\left(mod\,2\right)$,
\begin{eqnarray}
\textnormal{\textit{if}} & 1\leq n\leq2^{2-\alpha} & \textnormal{\emph{\textit{then}}\,\,}f\ge2^{\alpha-1}n\label{eq:33-1}\\
\textnormal{\textit{if}} & n>2^{2-\alpha} & \textnormal{\textit{then}\,}\, f\ge\left(3\times2^{\alpha-2}n-1\right)\label{eq:33-2}
\end{eqnarray}
for $D=\left(f{}^{2}-2^{\alpha}n\right)$ and $f\equiv0\left(mod\,2\right)$,
\begin{eqnarray}
\textnormal{\textit{if}} & \alpha\geq2\textnormal{\textnormal{\textit{, or if} \ensuremath{\alpha=0}\textit{or} \ensuremath{1}\textit{and} \ensuremath{3n\geq2^{3-\alpha}}}} & \textnormal{\textit{then}\,\,}f\ge2^{\alpha}n\label{eq:33-3}\\
\textnormal{\textit{if}} & \alpha=0\textnormal{ \textit{or} \ensuremath{1}\textit{and} }3n<2^{3-\alpha} & \textnormal{\textit{then}\,}\, f\geq\left(2^{\alpha-2}n+2\right)\label{eq:33-4}
\end{eqnarray}
for $D=\left(f{}^{2}-2^{\alpha}n\right)$ and $f\equiv1\left(mod\,2\right)$,
\begin{eqnarray}
\textnormal{\textit{if}} & n\geq\left(2^{2-\alpha}+1\right) & \textnormal{\textit{then}\,\,}f\geq3\times2^{\alpha-2}n\label{eq:33-5}\\
\textnormal{\textit{if}} & 1\leq n\leq2^{2-\alpha} & \textnormal{\textit{then}\,}\, f>2^{\alpha-3}n+3\label{eq:33-6}
\end{eqnarray}
\end{thm}
\begin{proof}
\noindent For squarefree $D$, $n$ and $r>0,\in\mathbb{Z}$; $\alpha\geq0,\in\mathbb{Z}$;
$f\left(u\right)$ any polynomial of $u\in\mathbb{Z}$ taking values
$f=f\left(u\right)>0,\in\mathbb{Z}$, let
\begin{equation}
f\equiv\left(2^{\alpha-2}n\right)\left(mod\,\left(2^{\alpha-1}n\right)\right)\label{eq:34}
\end{equation}
Obviously for $\alpha\geq3$, $f\equiv0\left(mod\,2\right)$ always,
while for $0\leq\alpha\leq2$, $f>0,\in\mathbb{Z}$ if $n\equiv0\left(mod\,2^{2-\alpha}\right)$
and $f\equiv0\left(mod\,2\right)$ if $n\equiv0\left(mod\,2^{3-\alpha}\right)$,
or $f\equiv1\left(mod\,2\right)$ if $n\equiv2^{2-\alpha}\left(mod\,2^{3-\alpha}\right)$. 

\noindent (i) For the case of the plus sign in front of $2^{\alpha}n$
in (\ref{eq:31}), i.e.
\begin{equation}
D=\left(f^{2}+2^{\alpha}n\right)\label{eq:35}
\end{equation}
one has 
\begin{equation}
f^{2}<f^{2}+2^{\alpha}n<\left(f+1\right)^{2}\label{eq:36}
\end{equation}
under the condition that
\begin{equation}
f\geq2^{\alpha-2}n\label{eq:37}
\end{equation}
It yields successively $a_{0}=\left\lfloor \sqrt{f^{2}+2^{\alpha}n}\right\rfloor =f$,
$b_{1}=f$, $c_{1}=2^{\alpha}n$; $a_{1}=\left(\frac{f-2^{\alpha-2}n}{2^{\alpha-1}n}\right)$,
$b_{2}=\left(f-2^{\alpha-1}n\right)$, $c_{2}=\left(f-2^{\alpha-2}n+1\right)$;
$a_{2}=1$ with the condition that 
\begin{equation}
f>2^{\alpha-2}n+1\label{eq:38}
\end{equation}
Then $b_{3}=(2^{\alpha-2}n+1)$, $c_{3}=(f+2^{\alpha-2}n-1)$; $a_{3}=1$,
$b_{4}=\left(f-2\right)$, $c_{4}=4$; $a_{4}=\left\lfloor \frac{f-1}{2}\right\rfloor $. 

\noindent (i.1) If $f\equiv0\left(mod\,2\right)$, $a_{4}=\left(\frac{f-2}{2}\right)$,
$b_{5}=\left(f-2\right)$, $c_{5}=\left(f+2^{\alpha-2}n-1\right)$;
$a_{5}=1$ with the condition (\ref{eq:38}). Then $b_{6}=\left(2^{\alpha-2}n+1\right)$,
$c_{6}=\left(f-2^{\alpha-2}n+1\right)$; $a_{6}=1$ with the condition
that 
\begin{equation}
f>3\times2^{\alpha-2}n-1\label{eq:39}
\end{equation}
and $b_{7}=\left(f-2^{\alpha-1}n\right)$, $c_{7}=2^{\alpha}n$; $a_{7}=\left(\frac{f-2^{\alpha-2}n}{2^{\alpha-1}n}\right)$,
$b_{8}=f$, $c_{8}=1$; $a_{8}=2f=2a_{0}$, yielding the continued
fraction 
\begin{equation}
\sqrt{f^{2}+2^{\alpha}n}=\left[f;\left(\frac{f-2^{\alpha-2}n}{2^{\alpha-1}n}\right),1,1,\left(\frac{f-2}{2}\right),1,1,\left(\frac{f-2^{\alpha-2}n}{2^{\alpha-1}n}\right),2f,\text{\dots}\right]\label{eq:40}
\end{equation}
with $r=7$, which by (\ref{eq:4}) yields
\begin{equation}
\frac{p_{7}}{q_{7}}=\left(\frac{\frac{f^{2}\left(f^{2}+2^{\alpha}n\right)}{2^{2\alpha-3}n^{2}}+1}{\frac{f\left(f^{2}+2^{\alpha-1}n\right)}{2^{2\alpha-3}n^{2}}}\right)\label{eq:42}
\end{equation}

\noindent (i.2) If $f\equiv1\left(mod\,2\right)$, $a_{4}=\left(\frac{f-1}{2}\right)$.
Then, $b_{5}=f$, $c_{5}=2^{\alpha-2}n$; $a_{5}=\left(\frac{f}{2^{\alpha-3}n}\right)$,
$b_{6}=f$, $c_{6}=4$; $a_{6}=\left(\frac{f-1}{2}\right)$; $b_{7}=\left(f-2\right)$,
$c_{7}=\left(f+2^{\alpha-2}n-1\right)$; $a_{7}=1$ with the condition
(\ref{eq:38}), $b_{8}=\left(2^{\alpha-2}n+1\right)$, $c_{8}=\left(f-2^{\alpha-2}n+1\right)$;
$a_{8}=1$ with the condition (\ref{eq:39}), $b_{9}=\left(f-2^{\alpha-1}n\right)$,
$c_{9}=2^{\alpha}n$; $a_{9}=\left(\frac{f-2^{\alpha-2}n}{2^{\alpha-1}n}\right)$,
$b_{10}=f$, $c_{10}=1$; $a_{10}=2f=2a_{0}$, yielding the continued
fraction
\begin{eqnarray}
\sqrt{f^{2}+2^{\alpha}n} & = & \left[f;\left(\frac{f-2^{\alpha-2}n}{2^{\alpha-1}n}\right),1,1,\left(\frac{f-1}{2}\right),\left(\frac{f}{2^{\alpha-3}n}\right),\right.\nonumber \\
 &  & \left.\left(\frac{f-1}{2}\right),1,1,\left(\frac{f-2^{\alpha-2}n}{2^{\alpha-1}n}\right),2f,\text{\dots}\right]\label{eq:43}
\end{eqnarray}
with $r=9$, which by (\ref{eq:4}) yields
\begin{equation}
\frac{p_{9}}{q_{9}}=\left(\frac{\frac{f^{2}\left(f^{2}+3\times2^{\alpha-2}n\right)^{2}}{2^{3\alpha-5}n^{3}}+1}{\frac{f\left(f^{2}+2^{\alpha-2}n\right)\left(f^{2}+3\times2^{\alpha-2}n\right)}{2^{3\alpha-5}n^{3}}}\right)\label{eq:45}
\end{equation}
Summarizing the conditions (\ref{eq:37}) to (\ref{eq:39}) for $f\equiv0\left(mod\,2\right)$
or $f\equiv1\left(mod\,2\right)$, one has that 

\noindent - if $1\leq n\leq2^{2-\alpha}$, then
\begin{equation}
f\ge2^{\alpha-1}n\label{eq:46}
\end{equation}

\noindent - if $n>2^{2-\alpha}$, then
\begin{equation}
f>3\times2^{\alpha-2}n-1\label{eq:47}
\end{equation}

\noindent (ii) For the case of the minus sign in front of $2^{\alpha}n$
in (\ref{eq:31}), i.e.
\begin{equation}
D=\left(f^{2}-2^{\alpha}n\right)\label{eq:48}
\end{equation}
one has
\begin{equation}
\left(f-1\right)^{2}<f^{2}-2^{\alpha}n<f^{2}\label{eq:49}
\end{equation}
giving the same condition as (\ref{eq:37}). It yields successively
$a_{0}=\left\lfloor \sqrt{f^{2}-2^{\alpha}n}\right\rfloor =\left(f-1\right)$,
$b_{1}=\left(f-1\right)$, $c_{1}=\left(2f-2^{\alpha}n-1\right)$;
$a_{1}=1$ with the condition that
\begin{equation}
f>2^{\alpha}n\label{eq:50}
\end{equation}
Then, $b_{2}=\left(f-2^{\alpha}n\right)$, $c_{2}=2^{\alpha}n$; $a_{2}=\left(\frac{f-3\times2^{\alpha-2}n}{2^{\alpha-1}n}\right)$,
$b_{3}=\left(f-2^{\alpha-1}n\right)$, $c_{3}=\left(f-2^{\alpha-2}n-1\right)$;
$a_{3}=2$ with the condition that
\begin{equation}
f>2^{\alpha-2}n+2\label{eq:51}
\end{equation}
and $b_{4}=\left(f-2\right)$, $c_{4}=4$; $a_{4}=\left\lfloor \frac{2f-3}{4}\right\rfloor $. 

\noindent (ii.1) If $f\equiv0\left(mod\,2\right)$, $a_{4}=\left(\frac{f-2}{2}\right)$.
Then, $b_{5}=\left(f-2\right)$, $c_{5}=\left(f-2^{\alpha-2}n-1\right)$;
$a_{5}=2$ with the condition that
\begin{equation}
f>3\times2^{\alpha-2}n\label{eq:52}
\end{equation}
Then, $b_{6}=\left(f-2^{\alpha-1}n\right)$, $c_{6}=2^{\alpha}n$;
$a_{6}=\left(\frac{f-3\times2^{\alpha-2}n}{2^{\alpha-1}n}\right)$,
$b_{7}=\left(f-2^{\alpha}n\right)$, $c_{7}=(2f-2^{\alpha}n-1)$;
$a_{7}=1$, $b_{8}=\left(f-1\right)$, $c_{8}=1$; $a_{8}=2\left(f-1\right)=2a_{0}$,
yielding the continued fraction
\begin{eqnarray}
\sqrt{f^{2}-2^{\alpha}n} & = & \left[\left(f-1\right);1,\left(\frac{f-3\times2^{\alpha-2}n}{2^{\alpha-1}n}\right),2,\left(\frac{f-2}{2}\right),\right.\nonumber \\
 &  & \left.2,\left(\frac{f-3\times2^{\alpha-2}n}{2^{\alpha-1}n}\right),1,2\left(f-1\right),\text{\dots}\right]\label{eq:53}
\end{eqnarray}
with $r=7$, which by (\ref{eq:4}) yields
\begin{equation}
\frac{p_{7}}{q_{7}}=\left(\frac{\frac{f^{2}\left(f^{2}-2^{\alpha}n\right)}{2^{2\alpha-3}n^{2}}+1}{\frac{f\left(f^{2}-2^{\alpha-1}n\right)}{2^{2\alpha-3}n^{2}}}\right)\label{eq:55}
\end{equation}
Summarizing the conditions (\ref{eq:37}) and (\ref{eq:50}) to (\ref{eq:52})
for $f\equiv0\left(mod\,2\right)$, one has that 

\noindent - if $\alpha\geq2$, or if $\alpha=0$ or $1$ and $3n\geq2^{3-\alpha}$,
then
\begin{equation}
f\geq2^{\alpha}n\label{eq:56}
\end{equation}

\noindent - if $\alpha=0$ or $1$ and $3n<2^{3-\alpha}$, then
\begin{equation}
f>2^{\alpha-2}n+2\label{eq:57}
\end{equation}

\noindent (ii.2) If $f\equiv1\left(mod\,2\right)$, $a_{4}=\left(\frac{f-3}{2}\right)$,
$b_{5}=\left(f-4\right)$, $c_{5}=\left(2f-2^{\alpha-2}n-4\right)$;
$a_{5}=1$ with the condition that
\begin{equation}
f\geq2^{\alpha-1}n+1\label{eq:58}
\end{equation}
and $b_{6}=\left(f-2^{\alpha-2}n\right)$, $c_{6}=2^{\alpha-2}n$;
$a_{6}=\left(\frac{f-2^{\alpha-2}n}{2^{\alpha-3}n}\right)$; $b_{7}=\left(f-2^{\alpha-2}n\right)$,
$c_{7}=\left(2f+2^{\alpha-2}n-4\right)$; $a_{7}=1$ with the condition
that
\begin{equation}
f\geq2^{\alpha-3}n+3\label{eq:59}
\end{equation}
Then $b_{8}=\left(f-4\right)$, $c_{8}=4$; $a_{8}=\left(\frac{f-3}{2}\right)$,
$b_{9}=\left(f-2\right)$, $c_{9}=\left(f-2^{\alpha-2}n-1\right)$;
$a_{9}=2$ with the condition (\ref{eq:52}), $b_{10}=\left(f-2^{\alpha-1}n\right)$,
$c_{10}=2^{\alpha}n$; 

\noindent $a_{10}=\left(\frac{f-3\times2^{\alpha-2}n}{2^{\alpha-1}n}\right)$,
$b_{11}=\left(f-2^{\alpha}n\right)$, $c_{11}=\left(2f-2^{\alpha}n-1\right)$;
$a_{11}=1$, $b_{12}=\left(f-1\right)$, $c_{12}=1$; $a_{12}=2\left(f-1\right)=2a_{0}$,
yielding the continued fraction 
\begin{eqnarray}
\sqrt{f^{2}-2^{\alpha}n} & = & \left[\left(f-1\right);1,\left(\frac{f-3\times2^{\alpha-2}n}{2^{\alpha-1}n}\right),2,\left(\frac{f-3}{2}\right),1,\right.\nonumber \\
 &  & \,\,\left(\frac{f-2^{\alpha-2}n}{2^{\alpha-3}n}\right),1,\left(\frac{f-3}{2}\right),2,\nonumber \\
 &  & \left.\left(\frac{f-3\times2^{\alpha-2}n}{2^{\alpha-1}n}\right),1,2\left(f-1\right),\text{\dots}\right]\label{eq:60}
\end{eqnarray}
with $r=11$, which by (\ref{eq:4}) yields
\begin{equation}
\frac{p_{11}}{q_{11}}=\left(\frac{\frac{f^{2}\left(f^{2}-3\times2^{\alpha-2}n\right)^{2}}{2^{3\alpha-5}n^{3}}-1}{\frac{f\left(f^{2}-2^{\alpha-2}n\right)\left(f^{2}-3\times2^{\alpha-2}n\right)}{2^{3\alpha-5}n^{3}}}\right)\label{eq:62}
\end{equation}
Summarizing the conditions (\ref{eq:37}), (\ref{eq:58}), (\ref{eq:59})
and (\ref{eq:52}) for $f\equiv1\left(mod\,2\right)$ and as $0\leq\alpha\leq2$,
one has that 

\noindent - if $n\geq2^{2-\alpha}+1$, then
\begin{equation}
f\geq3\times2^{\alpha-2}n\label{eq:63}
\end{equation}

\noindent - if $1\leq n\leq2^{2-\alpha}$, then
\begin{equation}
f>2^{\alpha-3}n+3\label{eq:64}
\end{equation}

\end{proof}
\noindent Like above, note that the conditions (\ref{eq:6-1-1}) and
$f>0$ means that $\exists\, k>0,\in\mathbb{Z}$ such that $f=2^{\alpha-2}\left(2k+1\right)n$,
yielding from (\ref{eq:31})
\begin{equation}
D=2^{\alpha}n\left(2^{\alpha-4}\left(2k+1\right)^{2}n\pm1\right)\label{eq:7-2-1}
\end{equation}
and for the plus sign in front of $1$ in (\ref{eq:7-2-1}), conditions
(\ref{eq:33-1}) and (\ref{eq:33-2}) reduce to $k\geq1$, and for
the minus sign in front of $1$ in (\ref{eq:7-2-1}), for $f\equiv0\left(mod\,2\right)$,
(\ref{eq:33-3}) and (\ref{eq:33-4}) reduce respectively to $k\geq0$
and $k\geq2$, and for $f\equiv1\left(mod\,2\right)$, (\ref{eq:33-5})
and (\ref{eq:33-6}) reduce respectively to $k\geq1$ and $k\geq2$.
Expressions simpler than (\ref{eq:32}) and (\ref{eq:33}) for the
convergents $\left(p_{r}/q_{r}\right)$ of $\sqrt{D}$ are then
\begin{equation}
\frac{p_{r}}{q_{r}}=\left(\frac{2^{\alpha-1}\left(2k+1\right)^{2}n\left(2^{\alpha-4}\left(2k+1\right)^{2}n\pm1\right)+1}{\left(2k+1\right)\left(2^{\alpha-3}\left(2k+1\right)^{2}n\pm1\right)}\right)\label{eq:7-3-1}
\end{equation}
if $f=2^{\alpha-2}\left(2k+1\right)n\equiv0\left(mod\,2\right)$ (i.e.
for either $\alpha\geq3$, or if $\alpha=0$, $n\equiv0\left(mod\,8\right)$,
or if $\alpha=1$, $n\equiv0\left(mod\,4\right)$, or if $\alpha=2$,
$n\equiv0\left(mod\,2\right)$), and
\begin{equation}
\frac{p_{r}}{q_{r}}=\left(\frac{2^{\alpha-3}\left(2k+1\right)^{2}n\left(2^{\alpha-2}\left(2k+1\right)^{2}n\pm3\right)\pm1}{\frac{1}{2}\left(2k+1\right)\left(2^{\alpha-2}\left(2k+1\right)^{2}n\pm1\right)\left(2^{\alpha-2}\left(2k+1\right)^{2}n\pm3\right)}\right)\label{eq:7-3-1-1}
\end{equation}
if $f=2^{\alpha-2}\left(2k+1\right)n\equiv1\left(mod\,2\right)$ (i.e.
for $\alpha\leq2$, if $\alpha=0$, $n\equiv4\left(mod\,8\right)$,
or if $\alpha=1$, $n\equiv2\left(mod\,4\right)$, or if $\alpha=2$,
$n\equiv1\left(mod\,2\right)$).

\noindent Several particular relations can also be deduced from the
general cases given in this Theorem.
\begin{cor}
\noindent For squarefree $D$, $d$, $m$, $n$ and $r>0,\in\mathbb{Z}$,
$d\equiv1\left(mod\,2\right)$; $\alpha\geq0,\in\mathbb{Z}$; and
for appropriate conditions as in Theorem 3, the following relations
hold for the convergents $\left(p_{r}/q_{r}\right)$ of $\sqrt{D}$
: 

\noindent (i) for $D=4d\left(d\pm2\right)$ and, in the case of the
minus sign, $d\geq3$,
\begin{equation}
\frac{p_{r}}{q_{r}}=\left(\frac{2d\left(d\pm2\right)+1}{d\pm1}\right)\label{eq:65}
\end{equation}

\noindent (ii) for $D=4\left(d^{2}\pm2\right)$ and, in the case of
the minus sign, $d\geq3$,
\begin{equation}
\frac{p_{r}}{q_{r}}=\left(\frac{2d^{2}\left(d^{2}\pm2\right)+1}{d\left(d^{2}\pm1\right)}\right)\label{eq:66}
\end{equation}

\noindent (iii) for $D=16n\left(nd^{2}\pm1\right)$ and, in the case
of the minus sign, $nd>1$,
\begin{equation}
\frac{p_{r}}{q_{r}}=\left(\frac{8nd^{2}\left(nd^{2}\pm1\right)+1}{d\left(2nd^{2}\pm1\right)}\right)\label{eq:67}
\end{equation}

\noindent (iv) for $D=4n\left(nd^{2}\pm2\right)$ and, in the case
of the minus sign, $nd>2$,
\begin{equation}
\frac{p_{r}}{q_{r}}=\left(\frac{2nd^{2}\left(nd^{2}\pm2\right)+1}{d\left(nd^{2}\pm1\right)}\right)\label{eq:68}
\end{equation}

\noindent (v) for $D=\left(d^{2}\pm4\right)$ and, in the case of
the minus sign, $d\geq3$
\begin{equation}
\frac{p_{r}}{q_{r}}=\left(\frac{\frac{d^{2}\left(d^{2}\pm3\right)^{2}}{2}\pm1}{\frac{d\left(d^{2}\pm1\right)\left(d^{2}\pm3\right)}{2}}\right)\label{eq:69}
\end{equation}

\noindent (vi) for $D=n\left(nd^{2}\pm4\right)$, $n\equiv1\left(mod\,2\right)$
and, in the case of the minus sign, $nd^{2}>2$,
\begin{equation}
\frac{p_{r}}{q_{r}}=\left(\frac{\frac{nd^{2}\left(nd^{2}\pm3\right)^{2}}{2}\pm1}{\frac{d\left(nd^{2}\pm1\right)\left(nd^{2}\pm3\right)}{2}}\right)\label{eq:70}
\end{equation}
\end{cor}
\begin{proof}
\noindent For squarefree $D$, $d$, $m$, $n$, $n'$ and $r>0,\in\mathbb{Z}$,
$d\equiv1\left(mod\,2\right)$, $\alpha$ and $\beta\geq0,\in\mathbb{Z}$;
let
\begin{equation}
f\left(d\right)=md^{\beta}\label{eq:30-1-1}
\end{equation}
such that $f\left(d\right)>0,\in\mathbb{Z}$. One has then the following
from Theorem 3, yielding the respective convergents by (\ref{eq:4}): 

\noindent (i) immediate for $m=2$, $\beta=1$, $\alpha=3$ and $n=d$
in (\ref{eq:30-1-1}) and (\ref{eq:32}).

\noindent (ii) immediate for $m=2$, $\beta=1$,$\alpha=3$ and $n=1$
in (\ref{eq:30-1-1}) and (\ref{eq:32}).

\noindent (iii) immediate for $m=2n$, $\beta=1$,$\alpha=3$ and
$n=2n'$ and dropping the prime sign $'$, in (\ref{eq:30-1-1}) and
(\ref{eq:32}).

\noindent (iv) immediate for $m=2n$, $\beta=1$ and $\alpha=3$ in
(\ref{eq:30-1-1}) and (\ref{eq:32}).

\noindent (v) immediate for $m=1$, $\beta=1$, $\alpha=2$ and $n=1$
in (\ref{eq:30-1-1}) and (\ref{eq:33}).

\noindent (vi) immediate for $m=n$, $\beta=1$ and $\alpha=2$ in
(\ref{eq:30-1-1}) and (\ref{eq:33}).
\end{proof}
\noindent For the cases (i) to (iv), it is easy to see that $D$ is
always $4\left(mod\,8\right)$. From case (iv), generalizing cases
(i) to (iii), $D=4n\left(nd^{2}\pm2\right)=\left(\left(2nd\right)^{2}\pm8n\right)$
and it is easy to see that, as $d\equiv1\left(mod\,2\right)$, if
$n\equiv0$ or $1\left(mod\,2\right)$ then $\left(2nd\right)\equiv0$
or $2\left(mod\,4\right)$, i.e. solutions will be found for even
integer squares $\left(2nd\right)^{2}$ different from powers of 2
($\left(2nd\right)^{2}\neq2^{\alpha}$), plus or minus even or odd
multiples of $8$ depending on the $0$ or $2\left(mod\,4\right)$
value of $2nd$. For example, for $n\equiv1\left(mod\,2\right)$,
$D=6^{2}\pm8$, $10^{2}\pm8$, $14^{2}\pm8$, ... and also $D=30^{2}\pm8$,
$30^{2}\pm24$, $30^{2}\pm40$, ...; for $n\equiv0\left(mod\,2\right)$,
$D=12^{2}\pm16$, $20^{2}\pm16$, $24^{2}\pm32$, $28^{2}\pm16$,
...

\noindent The case (v) of $D=\left(d^{2}+4\right)$ with $d\equiv1\left(mod\,2\right)$
is interesting as it gives the solutions for all the values of $D$
of the form $D=\left(8\triangle\left(k\right)+5\right)=\left(\left(2k+1\right)^{2}+4\right)$
with $k\geq0,\in\mathbb{Z}$, i.e. $D\equiv5\left(mod\,8\right)$
and $\frac{\left(D-5\right)}{8}=\triangle\left(k\right)$, where $\triangle\left(k\right)$
is the triangular number of $k$. Similarly, the case (v) of $D=\left(d^{2}-4\right)$
with $d\equiv1\left(mod\,2\right)$ gives the solutions for all the
values of $D$ of the form $D=\left(8\triangle\left(k\right)-3\right)=\left(\left(2k+1\right)^{2}-4\right)$
with $k\geq1,\in\mathbb{Z}$, i.e. $D\equiv-3\left(mod\,8\right)$
and $\frac{\left(D+3\right)}{8}=\triangle\left(k\right)$. For example,
$D=3^{2}+4$, $5^{2}\pm4$, $7^{2}\pm4$, $9^{2}\pm4$, ...

\noindent The case (vi) is also interesting as it generalizes the
previous case (v). $D$ can indistinctly be written as $D=n\left(nd^{2}\pm4\right)=\left(8\triangle\left(\frac{nd-1}{2}\right)+1\pm4n\right)$,
which is always $D\equiv5\left(mod\,8\right)$ as $d$ and $n\equiv1\left(mod\,2\right)$.
For example, $D=9^{2}\pm12$, $15^{2}\pm12$, $15^{2}\pm20$, $21^{2}\pm12$,
$21^{2}\pm28$, ...

\section{\noindent Conclusions}

\noindent The case considered in these two theorems of any polynomial
function $f\left(u\right)$ of $u\in\mathbb{Z}$ can easily be extended
to any function $g\left(x\right)$ of $x\in\mathbb{\mathbb{R}}$ or
$h\left(z\right)$ of $z\in\mathbb{\mathbb{C}}$ as long as $g\left(x\right)$
or $h\left(z\right)>0,\in\mathbb{Z}$ for $\forall x\in\mathbb{\mathbb{R}}$
or $\forall z\in\mathbb{\mathbb{\mathbb{C}}}$. 

\noindent These two general theorems and the various relations in
the two corollaries allow the immediate calculation of the fundamental
solutions of the simple Pell equation for most of the values of $D$.
A quick survey shows that these two theorems provide solutions for
all squarefree integer values of $D$ for $D\leq10$, for 72.2\% of
the values of $D$ for $D<10^{2}$, for 35.4\% for $D\leq10^{3}$,
and for 15.7\% for $D<10^{4}$.

\end{document}